\documentclass[11pt]{amsart}

\usepackage{amsmath}               
\usepackage{amsthm}
\usepackage{amssymb, amsthm, verbatim,enumerate,bbm,color} 
\usepackage{indentfirst}
\usepackage{enumitem}
\usepackage{mathtools}
\usepackage{floatrow}
\usepackage{enumitem}
\usepackage[utf8]{inputenc}
\usepackage{hyperref}
\usepackage{cleveref}
\usepackage{xcolor}
\usepackage{mathtools}
\usepackage[toc,page]{appendix}
\hypersetup{linktocpage}

\title{ Overcrowding for zeros of Hyperbolic Gaussian analytic functions }

\usepackage{graphicx}

\allowdisplaybreaks
\usepackage[a4paper, hmargin=2.5cm, vmargin={3cm, 3cm}]{geometry}
\usepackage{fancyhdr}

\theoremstyle{plain}
\newtheorem{theorem}{Theorem}[section]
\newtheorem*{theorem*}{Theorem}
\newtheorem{lemma}[theorem]{Lemma}
\newtheorem{claim}[theorem]{Claim}
\newtheorem*{assumption*}{Assumption}

\newtheorem{corollary}[theorem]{Corollary}

\newtheorem{remark}[theorem]{Remark}

\newcommand{\Var}{\ensuremath{\textrm{Var}}}

\DeclareMathOperator{\Li}{Li}

\RequirePackage[normalem]{ulem} 
\RequirePackage{color}\definecolor{RED}{rgb}{1,0,0}\definecolor{BLUE}{rgb}{0,0,1} 

\title{Overcrowding for zeros of Hyperbolic Gaussian analytic functions}

\author{Keren Mor Waknin}
\address{School of Mathematical Sciences, Tel Aviv University, Tel Aviv 69978, Israel}
\email{kerenmor22@gmail.com}

\begin{document}
\maketitle

\begin{abstract}
	We consider the family $\{f_L\}_{L>0}$ of Gaussian analytic functions in the unit disk, distinguished by the invariance of their zero set with respect to hyperbolic isometries. Let $n_L\left(r\right)$ be the number of zeros of $f_L$ in a disk of radius $r$.
	We study the asymptotic probability of the rare event where there is an overcrowding of the zeros as $r\uparrow1$, i.e. for every $L>0$, we are looking for the asymptotics of the probability $\mathbb{P}\left[n_L(r)\geq V(r)\right]$ with $V\left(r\right)$ large compared to the $\mathbb{E}\left[n_L\left(r\right)\right]$.
	Peres and Virág showed that for $L=1$ (and only then) the zero set forms a determinantal point process, making many explicit computations possible.
	Curiously, contrary to the much better understood planar model, it appears that for $L<1$ the exponential order of decay of the probability of overcrowding when $V$ is close to $\mathbb{E}\left[n_L\left(r\right)\right]$ is much less than the probability of a deficit of zeros.
\end{abstract}

\section{Introduction}

We study hyperbolic Gaussian analytic functions (GAFs) in the unit disk $\mathbb{D}$
\begin{equation*}
f_L\left(z\right)=\sum_{n\geq0}\xi_n\sqrt{\frac{L\left(L+1\right)\cdots\left(L+n-1\right)}{n!}}z^n, \quad \text{for} \ L>0,
\end{equation*}
where $\{\xi_n\}_{n\geq0}$ are independent standard complex Gaussian random variables.
The zero set of this family has a distinguished property, which is the invariance of the distribution of their zero set $\mathcal{Z}_f=f^{-1}\{0\}$ with respect to M\"{o}bius transformations of the unit disk. This property is unique in the sense that these are the only GAFs, up to a multiplication by a deterministic nowhere vanishing analytic function, whose zeros are isometry-invariant under the respective group of isometries \cite[Chap.~2]{gaf_book}. 
Let $n_L\left(r\right)=\#\{z \ : \ f\left(z\right)=0 \ \text{and} \ \left|z\right|\leq r\}$ be the number of zeros of $f_L$ in the disk of radius $r\in\left(0,1\right)$.
Because of the invariant distribution of the zero set, the first intensity is equal to a constant multiple of the hyperbolic measure on the unit disk. 
In fact, it is not hard to check that 
\begin{equation}\label{expactation}
	\mathbb{E}[n_L(r)]=\frac{Lr^2}{1-r^2}
\end{equation}
see  \cite{Edelman-kostlan}. Buckley \cite{Vardisk} studied the asymptotic of the variance $v_L\left(r\right)$ and proved that there is a transition in the asymptotics of the variance at $L=\frac{1}{2}$.
The case $L=1$ is special in the hyperbolic models.
Peres and Vir\'{a}g \cite{Determinantal} discovered that the zero set of the Gaussian Taylor series $f_1(z)=\sum_{n=0}^{\infty}\xi_nz^n$ is a \emph{determinantal} point process, which allowed them to directly compute a lot of its characteristics (see \cite[Chapter~4]{gaf_book} and also \cite[Section~5.1]{gaf_book}).
Peres and Vir\'{a}g also showed that $L=1$ is the only instance when the zero set of $f_L$ is a determinantal point process.
In \cite{holedisk} Buckley, Nishry, Peled, and Sodin found the rate of decay of the hole probability $ p_H\left(r\right):= \mathbb{P}\left[\{f_L\neq0 \ \text{on} \ r\bar{\mathbb{D}}\}\right]$ as $r\uparrow1$. 
Recently, Buckley and Nishry \cite{CLT-Hole} proved that the limiting (in law) behavior of the normalized version of $n_L\left(r\right)$:
\begin{equation*}
\widehat{n}_L\left(r\right):=\frac{n_L\left(r\right)-\mathbb{E}\left[n_L\left(r\right)\right]}{\sqrt{v_L\left(r\right)}}, \quad r\uparrow 1,
\end{equation*}  has a transition when $L=\frac{1}{2}$. They showed that $n_L\left(r\right)$ satisfies a CLT for $L\geq\frac{1}{2}$ and exhibits non-Gaussian limiting behavior for $L<\frac{1}{2}$. The case $L=1$ was treated previously by Peres and Vir\'{a}g \cite{Determinantal}.
For the hyperbolic Gaussian analytic functions Krishnnapur \cite{overcrowding} found some bounds for the probability of the overcrowding event $\{n_L\left(r\right)\geq m\}$ with $r$ fixed as $m\rightarrow\infty$.
In this work we give matching lower and upper bounds for the overcrowding probability when $r\uparrow1$ (see \Cref{Theorem: overcrowding}). 

	The first result is about large deviations in the case $L=1$.
 \begin{theorem}\label{Theorem: babycase}
	Let $\mu_r,v_1\left(r\right)$ be the expectation and the variance of $n_1\left(r\right)$ respectively. Then, for $r\uparrow1$,
	\begin{equation*}
	-\log\mathbb{P}\left[\left|n_1(r)-\mu_r\right|\geq t\cdot v_1\left(r\right)^{\alpha}\right] \sim c(t)v_1(r)^{2\alpha-1}
	\end{equation*}
	where
	\small{
		\begin{equation*}
		c(t)=
		\begin{cases}
		{\frac{t^2}{2}} & \frac{1}{2}<\alpha<1, \\
		\tfrac{1}{2}\left(t+2\right)^2 + 2\Li_2\left(1 - \exp \left(\frac{t+2}{2} + W_0\left(-\frac{t+2}{2e^{\frac{t+2}{2}}}\right)\right)\right) + 
		\left(t+2\right)W_0\left(-\frac{t+2}{2e^{\frac{t+2}{2}}}\right) & \alpha=1, \\ 
		{\frac{t^2}{4}} & \alpha>1,
		\end{cases}
		\end{equation*}}
	$\Li_2(z)=-\int_{0}^{z}\frac{\log(1-u)}{u}du$ is the dilogarithm function \cite{dilofuction}, and $W_0$ is the principal branch of the Lambert $W$ function \cite{Wfuction}.
	
\end{theorem}
We prove \Cref{Theorem: babycase} in \Cref{proofbabycase}.
Our second result is the asymptotics of the overcrowding probability for general $L>0$. 
{
	\begin{theorem}
		\label{Theorem: overcrowding}
		Fix $L>0$. Then there exists $C_L>0$ such that for all $r\in\left[\tfrac{1}{2},1\right)$ and all $V>1$ satisfying  
		\begin{equation*}
		V\geq\frac{C_L}{1-r}\log \frac{1}{1-r},
		\end{equation*}   
		we have
		\begin{equation*}
		-\log\mathbb{P}\left[n_L(r)\geq V\right]\simeq_L\left(1-r\right)V^2.
		\end{equation*}
	\end{theorem}
	\begin{corollary}
		For all $L>0$ and $\alpha>1$ we get as $r\uparrow1$,
		\begin{equation*}
		-\log\mathbb{P}\left[n_L(r)\geq v_L\left(r\right)^{\alpha}\right]\simeq_{L,\alpha}\begin{cases}
		\left(\frac{1}{1-r}\right)^{2\alpha-1} & L>\frac{1}{2},\\
		\left(\frac{1}{1-r}\right)^{2\alpha-1}\log^{2\alpha}\left(\frac{1}{1-r}\right) & L=\frac{1}{2}, \\
		\left(\frac{1}{1-r}\right)^{2\alpha(2-2L)-1} & 0<L<\frac{1}{2}.
		\end{cases}
		\end{equation*}	
	\end{corollary}
\newpage
\subsection{Related works}

\subsubsection{Comparison with the Gaussian entire function (GEF)}
An analogue process to $f_L$ whose zero set distribution is invariant under isometries of the plane is given by 	\begin{equation*}
g_L\left(z\right)=\sum_{n\geq0}\xi_n\sqrt{\frac{L^n}{n!}}z^n.
\end{equation*} with $L>0$. Note that it is enough to consider the case $L=1$ since other values of $L$ are just scaled versions of this case (unlike the hyperbolic model).
The mean number of zeros is equal to $\frac{L}{\pi}$, with respect to the Lebesgue measure $dm\left(z\right)$ on the plane. The asymptotics of the variance in the planar $v\left(R\right)$ was computed by Forrester and Honner in \cite{Varplane}: \begin{equation*}
v\left(R\right)=cR+o\left(R\right), \quad  R\rightarrow\infty.
\end{equation*} 
We use upper case $R$ for the radius to emphasize that the function's domain is $\mathbb{C}$.
In \cite{Vardisk}, Buckley proved that for the hyperbolic model the asymptotic of the variance is given by:
\begin{equation*}
v_L\left(r\right)\sim 
\begin{cases}
\frac{c_L}{1-r}  & \text{for }\  L>\frac{1}{2},\\
\frac{c}{1-r}\log\frac{1}{1-r}  & \text{for} \ L=\frac{1}{2},\\
\frac{c_L}{(1-r)^{2-2L}}  & \text{for } 0<L<\frac{1}{2},
\end{cases} r\rightarrow1^{-},
\end{equation*}
where the constants are explicitly computed and positive in all three cases.
Note that in the planar case, the expectation $\mathbb{E}\left[n\left(R\right)\right]$ is in fact the Lebesgue measure of the disc with radius $R$, which is much larger (by factor of $R$) than the variance $v\left(R\right)$, which is proportional to the length of the boundary of the disc. However, in the hyperbolic case, the hyperbolic area and the length circumference of a disk of radius $r$ are approximately the same and contrary to the plane, the variance is greater or equal (up to constants) than the expectation \eqref{expactation} (depending on the values of $L$). 
It is also quite natural to study the probability of (large) fluctuations. We will restrict ourselves to the case of a disk with large radius $r$ with respect to the function's domain (i.e in the unit disk $r$ is close to 1 and in the complex plane $R$ tends to infinity) and study the probability:
\begin{equation*}
\mathbb{P}\left[\left|n_L\left(r\right)-\mathbb{E}\left[n_L\left(r\right)\right]\right|\geq  v_L\left(r\right)^{\alpha}\right],
\end{equation*}
beyond what is known by asymptotic normality (i.e. $\alpha>\frac{1}{2}$). This might be called moderate$\backslash$large$\backslash$huge deviations in probability language, depending on $\alpha$ and the model we consider.
For the GEF there is a complete answer in terms of $\alpha$ given by:

\begin{theorem}[\emph{Jancovici-Lebowitz-Manificat Law}
	 {\cite{jlm,overcrowding,holeplane}}]
	As $R\rightarrow\infty$
	\begin{equation*}
	\frac{\log\left(-\log\left(\mathbb{P}\left[\left|n\left(R\right)-R^2\right|>R^{\alpha}\right]\right)\right)}{\log R}\rightarrow\begin{cases}
	2\alpha-1 &  \frac{1}{2}\leq\alpha\leq 1,\\
	3\alpha-2 &  1\leq\alpha\leq 2,\\
	2\alpha & \alpha\geq 2.
	\end{cases}
	\end{equation*}
\end{theorem}
		Here we can see that the asymptotic behavior changes for different values of the exponent $\alpha$, unlike the case of the hyperbolic GAF \Cref{Theorem: babycase}.
		i.e. considering the $\log\log$ asymptotic of this probability as $r\uparrow1$, it holds that the power remains $2\alpha-1$. 

\subsubsection{Determinantal point processes in the unit disk}

		Our result for the determinantal process in the hyperbolic model is related to recent work of Fenzl and Lambert (see \cite{Determinan}), where they compute the precise deviations for disk counting statistic of rotation invariant determinantal point processes in the disk and the plane. In particular they proved that the large deviations principle holds for the determinantal point process ($L=1$) and provided a general formula for the rate of decay in terms of the rate function. Here we provide an explicit expression to the rate function. Since the only case when the hyperbolic GAF form a determinantal point process is when $L=1$, we need to use other techniques for different values of $L$.

\subsubsection{Overcrowding estimates.}
		In \cite{overcrowding} Krishnapur studied the overcrowding probability for zeros of the hyperbolic GAF. Our lower bound for $\mathbb{P}\left[n_L\left(r\right)\geq V\right]$ is close to his. Our upper bound is more precise, mainly because we found a more accurate estimate of the asymptotic probability  
		\begin{equation*}
		\mathbb{P}\left[\int_{r\mathbb{T}}\log|f(re^{i\theta}	)|d\mu\leq -x \right], \qquad \text{as} \ \, x\to \infty.
		\end{equation*}
		Krishnapur conjectured that his lower bound gives the correct order of decay of the overcrowding probability, and our \Cref{Theorem: overcrowding} confirms his conjecture. 
\subsubsection{Hole probability}
Consider the event when $f\left(z\right)\neq0$ in $r\mathbb{D}$, sometimes called the \emph{hole event}.  
We denote $ p_H\left(r\right):= \mathbb{P}\left[\{f_L\neq0 \ \text{on} \ r\bar{\mathbb{D}}\}\right]$ and are interested in its decay as $R\rightarrow\infty$ or $r\rightarrow1^{-}$. 
There are results regarding the rate of decay of the hole probability in both models - the planar and the hyperbolic: 
\begin{theorem}[{\cite{Hole}},\cite{Hole gosh},\cite{holeplane}]
	Suppose $R\rightarrow\infty$. Then, 
	\begin{equation*}
	-\log p_H\left(R\right)=\frac{e^2}{4} R^4+O\left(R^2\log^2R\right).
	\end{equation*}
\end{theorem}
\begin{theorem}[{\cite[Theorem~1]{holedisk}}]
	Suppose $r\uparrow1$. Then, 
	\begin{equation*}
		-\log p_H(r)\simeq_L \begin{cases}
		\frac{1}{\left(1-r\right)^L}\log\frac{1}{1-r}  & \text{for }\  0<L<1,\\
		\frac{1}{1-r}\log\frac{1}{1-r}  & \text{for} \ L=1,\\
		\frac{1}{1-r}\log^2\frac{1}{1-r}  & \text{for } L>1,
		\end{cases} r\rightarrow1^{-},
	\end{equation*}
	where the constants are explicitly computed and positive.
\end{theorem}
\subsubsection{Loss of the balance between excess of zeros and their deficit.}
	Consider the event of excess of zeros where $\{n_L\left(r\right)>\mu_r + v_L\left(r\right)\}$ and the event of deficit of zeroes $\{n_L\left(r\right)<\mu_r-v_L\left(r\right)\}$ where $\mu_r=\mathbb{E}\left[n_L\left(r\right)\right]$. 
		Heuristically, letting $\alpha\downarrow1$ in \Cref{Theorem: overcrowding} would imply 
		\begin{align*}
		-\log\mathbb{P}\left[n_L\left(r\right)>\mu_r+v_L\left(r\right)^{\alpha}\right]\simeq -\log\mathbb{P}\left[n_L\left(r\right)>v_L\left(r\right)^{\alpha}\right]&\simeq 
		v_L\left(r\right)^{2\alpha-1}\\&\overset{\alpha\downarrow 1}{\sim}
		\begin{cases}
		v_L\left(r\right) & \frac{1}{2}\leq L <1, \\
		v_L\left(r\right)^{\frac{3-4L}{2-2L}} & 0< L<\frac{1}{2},
		\end{cases}
		\end{align*} 
		while the hole probability results implies that
		\begin{align*}
		-\log\mathbb{P}\left[n_L\left(r\right)<\mu_r-v_L\left(r\right)^{\alpha}\right]&\geq\frac{1}{\left(1-r\right)^L}\log\frac{1}{1-r}
		 \simeq 
		\begin{cases}
		v_L\left(r\right)^{L}, & \frac{1}{2}\leq L<1,\\
		v_L\left(r\right)^{\frac{L}{2-2L}}, & 0<L<\frac{1}{2}
		\end{cases}
		\end{align*}
		up to some logarithmic factor which we neglect. 
		I.e. the exponent is different and we do not get the symmetry between these events (contrary to the planar model), the probability for deficit is bigger.
}

	\subsection{Notation \& Some special functions} \label{special functions}
		\begin{itemize}[label={--}]
			\item $c$ and $C$ denote positive numerical constants that do not depend on $r,\alpha,L$. By $c_L$ we denote positive constant that depends only on $L$. In this work, the values of these constants are irrelevant and may vary from line to line.
			\item The notation $X\simeq Y$ means that there exists positive constants such that $cX\leq Y\leq CX$. The notation $X\simeq _LY$ is used when the constants depend on $L$. 
			\item $e\left(t\right)=e^{2\pi it}$.
			\item $\left[x\right]$ denotes the integer part of $x$.
			\item $n\equiv k\left(N\right)$ means that $n\equiv k\ \text{modulo} \ N$.
			\item $r\mathbb{D}$ denotes the open disk $\{z :|z|<r\}$ and $r\mathbb{T}$ denotes its boundary $\{z : |z|=r\}$.
			\item The planar Lebesgue measure is denoted by $m$ and the normalized angular measure on $r\mathbb{T}$ is denoted by $\mu$.
			\item We set $\delta:=1-r$. We always assume that $r$ is sufficiently close to $1$ (and hence $\delta$ close to $0$). 
			\item Let $f\left(r\right),g\left(r\right)$ be positive functions. The notation $f\left(r\right)\sim g\left(r\right)$ as $r\uparrow1$ means that $\lim_{r\uparrow 1}\frac{f\left(r\right)}{g\left(r\right)}=1$. The notation $f\left(r\right)\ll g\left(r\right)$ as $r\uparrow1$ means that $\lim_{r\uparrow1}\frac{f\left(r\right)}{g\left(r\right)}=0$. 
		\end{itemize}

\subsection*{Acknowledgments}

I am deeply grateful to my MSc advisors, Alon Nishry and Mikhail Sodin, for their constant encouragement, endless patience and enlightening discussions throughout this project. I would also like to thank Naomi Feldheim, Manjunath Krishnapur, Ron Peled and Oren Yakir for their valuable comments and fruitful
discussions. I thank Lakshmi Priya for reading this paper with great care and making
very insightful comments, which improved both the mathematics and the presentation. I thank the referee for careful reading this work and several useful suggestions. 
This work was supported by ERC Advanced Grant 692616, ISF Grant 382/15 and by BSF Start up Grant no. 2018341.

\section{Proof of Theorem 1.2}
{Fix $L>0$. Recall that we want to prove that there exists $C_L>0$ such that for all $r\in\left[\tfrac{1}{2},1\right)$ and all $V>1$ satisfying 
	\begin{equation}\label{assumption}
	V\geq\frac{C_L}{1-r}\log \frac{1}{1-r},
	\end{equation}   
	we have
	\begin{equation*}
	-\log\mathbb{P}\left[n_L(r)\geq V\right]\simeq_L\left(1-r\right)V^2.
	\end{equation*}
\subsubsection*{Several remarks on \Cref{Theorem: overcrowding}}
\begin{enumerate}
	\item {Assumption \eqref{assumption} is meant to ensure that the number of zeros $V$ is large enough (compared with the expectation). 
	\item 
	Recall that
	\begin{equation*}
	v_L\left(r\right)=\Var\left[n_L\left(r\right)\right]\sim 
	\begin{cases}
	\frac{c_L}{1-r}  & \text{for }\  L>\frac{1}{2},\\
	\frac{c}{1-r}\log\frac{1}{1-r}  & \text{for} \ L=\frac{1}{2},\\
	\frac{c_L}{(1-r)^{2-2L}}  & \text{for } 0<L<\frac{1}{2}
	\end{cases} \quad ,\quad r\uparrow1.
	\end{equation*} 
	By the above, $v_L^\alpha\left(r\right)$ will satisfy assumption \eqref{assumption}, provided that either
	 $L>0$, $\alpha>1$, or $L<\frac{1}{2}$, $\alpha\in (\frac{1}{2-2L},1]$.
	In both cases $v_L(r)^\alpha\gg\mu_r=\frac{Lr^2}{1-r^2}$ as $r\uparrow1$, and hence 
	\begin{equation*}
	\log\mathbb{P}\left[\left|n_L(r)-\mu_r\right|\geq v_L\left(r\right)^{\alpha}\right]\simeq\log\mathbb{P}\left[n_L\left(r\right)\geq v_L\left(r\right)^\alpha\right].
	\end{equation*}} 
\end{enumerate}
In order to prove \Cref{Theorem: overcrowding} we find matching lower and upper bounds for the $(\log)$ probability of the overcrowding event. We employ strategies somewhat similar to the ones introduced in {\cite{overcrowding,holeplane,holedisk,Beta*}. {Throughout the proof we will assume that $r$ is sufficiently close to $1$, as we can do.} 

\subsection{Lower Bound} 
\begin{claim}
	\label{claim:lower bound}
	For $r\in[\tfrac{1}{2},1)$ we have
	$$\log\mathbb{P}\left[n_L\left(r\right)\geq V\right]\geq -c_L\left(1-r\right)V^2.$$
\end{claim}
	\begin{proof}
		Fix $r$ close to 1 and take $m=m\left(r\right)=\left[V\right]+1$.
		Consider the following independent events:
		\begin{itemize}
			\item $E_1\coloneqq\left\{\left|\xi_n\right|\leq \sqrt{n} \ ,\ n\geq2m+1\right\},$
			\item $E_2\coloneqq\left\{\left|\xi_n\right|\binom{-L}{n}^{\frac{1}{2}}\leq\binom{-L}{m}^{\frac{1}{2}}r^m \ , \ m+1\leq n\leq 2m\right\}, $
			\item $E_3\coloneqq\left\{\left|\xi_m\right|\geq 4m\sqrt{1-r}\right\},$
			\item $E_4\coloneqq\left\{\left|\xi_n\right|\binom{-L}{n}^{\frac{1}{2}}r^n\leq\sqrt{1-r}\binom{-L}{m}^{\frac{1}{2}}r^m \ , \ 0\leq n\leq m-1\right\}$.
		\end{itemize}
			\begin{claim}
				\label{claim: Rosh}
			Under the events $E_1,E_2,E_3,E_4$, on $|z|=r$ we have,
			\begin{equation*}
			\left|f_L(z)-\xi_m\binom{-L}{m}^{\frac{1}{2}}z^m\right|<\left|\xi_m\binom{-L}{m}^{\frac{1}{2}}z^m\right|
			\end{equation*}
			whence , $n_L\left(r\right)=m$. 
			\end{claim}
			\begin{claim}
				\label{claim: Product prob}
				 For all $1\leq j\leq 4$, 
				 \begin{equation*}
				 \mathbb{P}\left[E_j\right]\geq \exp\left(-c_L\left(1-r\right)V^2\right).
				 \end{equation*}
			\end{claim}
				Overall, using the independence of these events we get the desired lower bound, proving \Cref{claim:lower bound}.
		It remains to prove these two claims.
				\begin{proof}[Proof of \Cref{claim: Rosh}] It holds that
					\begin{align*}
					\left|f_L\left(z\right)-\xi_m\binom{-L}{m}^{\frac{1}{2}}z^m\right|&\leq\left(\sum_{n=0}^{m-1}+\sum_{n=m+1}^{2m}+\sum_{n=2m+1}^{\infty}\right)\left|\xi_n\right|\binom{-L}{n}^{\frac{1}{2}}r^n
					\\&\leq m\sqrt{1-r}\binom{-L}{m}^{\frac{1}{2}}r^m+\binom{-L}{m}^{\frac{1}{2}}r^m\cdot\frac{r^{m+1}(1-r^m)}{1-r}\\&+\sum_{n=2m+1}^{\infty}\sqrt{n}\binom{-L}{n}^{\frac{1}{2}}r^n.
					\end{align*}
					So we need to show that the second and the third terms on the right hand side are also bounded by $m\sqrt{1-r}\binom{-L}{m}^{\frac{1}{2}}r^m$. 
					\begin{itemize}[label={--}]
						\item \underline{2nd term:} 
						\begin{equation*}
						r^m=e^{-m\left|\log r\right|}\ll\sqrt{1-r} \quad \text{(assumption \eqref{assumption})}
						\end{equation*}
						whence,
						\begin{equation*}
						r^{m+1}\frac{1-r^m}{1-r}<mr^m\ll m\sqrt{1-r}.
						\end{equation*}
						\item \underline{3rd term:} Let $c_n=\sqrt{n}\binom{-L}{n}^{\frac{1}{2}}r^n$. Then $\frac{c_{n+1}}{c_n}=\sqrt{\frac{L+n}{n}}\cdot r\rightarrow r$ as $n\rightarrow\infty$. Choose $r$ sufficiently close to 1 such that $\frac{c_{n+1}}{c_n}\leq\frac{1}{2}\left(1+r\right)$. Then, 
						\begin{equation*}
						\sum_{n=2m+1}^{\infty}c_n\leq c_{2m+1}\left(\frac{1+r}{2}\right)^{2m+1}\frac{1}{1-\frac{1+r}{2}}\ll c_{2m+1} \ \text{(assumption \eqref{assumption})}
						\end{equation*}
						\begin{equation*}
						\lesssim c_mr^m \ \left(\text{since} \binom{-L}{2m+1}\simeq\binom{-L}{m}\right)
						\end{equation*}
						\begin{equation*}
						=\sqrt{m}r^m\binom{-L}{m}^{\frac{1}{2}}r^m\ll m\sqrt{1-r}\binom{-L}{m}^{\frac{1}{2}}r^m.
						\end{equation*}
						\end{itemize}
				\end{proof}
		\begin{proof}[Proof of \Cref*{claim: Product prob}] 
					First note that  $|\xi_n|^2\sim\exp(1)$. 
					We start with the event $E_1$. \\
					As $r\uparrow1:$
					\begin{align*}
					\mathbb{P}\left[\left|\xi_n\right|\leq\sqrt{n}, \forall n\geq2m+1\right]&\geq 1-\sum_{n=2m+1}^{\infty}e^{-n}\\&=1-e^{-\left(2m+1\right)}\frac{1}{1-e^{-1}}\geq\frac{1}{2}.
					\end{align*}
					For $E_2$, since $\binom{-L}{n}\underset{n\rightarrow\infty}{\sim}\frac{n^{L-1}}{\Gamma(L)}$, we get that, for $m+1\leq n\leq 2m$, it holds that $\frac{\binom{-L}{m}^{\frac{1}{2}}}{\binom{-L}{n}^{\frac{1}{2}}}\simeq1$
					Therefore,
					\begin{equation*}
					 \frac{\binom{-L}{m}^{\frac{1}{2}}}{\binom{-L}{n}^{\frac{1}{2}}}\cdot r^m\ll 1.
					\end{equation*}
					Next, 
						\begin{align*}
						\mathbb{P}\left[\left|\xi_n\right|\leq\frac{\binom{-L}{m}^{\frac{1}{2}}}{\binom{-L}{n}^{\frac{1}{2}}}\cdot r^m \ ,\forall \ m+1\leq n\leq2m\right]&\geq\prod_{n=m+1}^{2m}c_Lr^{2m}\\&=c_L^mr^{2m^2}\geq e^{-c_Lm^2\left(1-r\right)}. 
						\end{align*}
					For $E_3$, $\mathbb{P}\left[|\xi_m|\geq 4m\sqrt{1-r}\right]=e^{-16m^2\left(1-r\right)}$.
					Finally, for $E_4$ we have,
					\begin{align*}
						\mathbb{P}\left[\left|\xi_0\right|\leq\sqrt{1-r} \frac{\binom{-L}{m}^{\frac{1}{2}}}{\binom{-L}{n}^{\frac{1}{2}}}\cdot r^{m-n}\right]&\geq\mathbb{P}\left[\left|\xi_0\right|\leq\sqrt{1-r}\cdot m^{\frac{L-1}{2}}r^m\right]
						\\&
						\geq \tfrac{1}{2}\left(1-r\right)m^{L-1}r^{2m}
						\geq e^{-c_Lm\left(1-r\right)},
						\end{align*}
						and for $1\leq n\leq m-1$ we have $\binom{-L}{n}\leq B_Ln^{L-1}$, while $\binom{-L}{m}\geq b_Lm^{L-1}$ hence
						\begin{align*}
						\mathbb{P}\left[\left|\xi_n\right|\leq\sqrt{1-r}\frac{\binom{-L}{m}^{\frac{1}{2}}}{\binom{-L}{n}^{\frac{1}{2}}} r^{m-n}\right]\geq\mathbb{P}\left[\left|\xi_n\right|\leq A_L\sqrt{1-r}\left(\frac{m}{n}\right)^{\frac{L-1}{2}} r^{m-n}\right]
						\end{align*}
						Note that for $r$ sufficiently close to 1, 
						\begin{equation*}
						\max_{1\leq n\leq m-1}\left(\frac{m}{n}\right)^{\frac{L-1}{2}}r^{m-n}=\left(\frac{m}{m-1}\right)^{\frac{L-1}{2}}r\leq D_L
						\end{equation*} 
						which implies that $A_L\sqrt{1-r}\left(\frac{m}{n}\right)^{\frac{L-1}{2}}r^{m-n}\ll1$, and hence, 
						\begin{equation*}
						\mathbb{P}\left[\left|\xi_n\right|\leq A_L\sqrt{1-r}\left(\frac{m}{n}\right)^{\frac{L-1}{2}}\cdot r^{m-n}\right]\geq c_L\left(1-r\right)\left(\frac{m}{n}\right)^{L-1}r^{2\left(m-n\right)}
						\end{equation*}
						So finally we get, 
						\begin{equation*}
						\mathbb{P}\left[\left|\xi_n\right|\binom{-L}{n}^{\frac{1}{2}}r^n\leq\sqrt{1-r}\binom{-L}{m}^{\frac{1}{2}}r^m \ ,\forall \ 0\leq n\leq m-1 \right]
						\end{equation*}
						\begin{align*}
						&
						\geq e^{-c_Lm\left(1-r\right)}\prod_{n=1}^{m-1}\left(d_L\left(1-r\right)\left(\frac{m}{n}\right)^{L-1}r^{2\left(m-n\right)}\right)
						\\&
						=e^{-c_Lm\left(1-r\right)}d_L^{m}\left(1-r\right)^mr^{b_Lm^2}\prod_{n=1}^{m-1}\left(\frac{m}{n}\right)^{L-1}
						\geq e^{-c_Lm^2\left(1-r\right)},
						\end{align*}		
					proving \Cref{claim: Product prob}.
			\end{proof}
		Combining altogether, we get the desired lower bound.
			\end{proof}
\subsection{Upper Bound}
\begin{claim}
	For $r\in[\tfrac{1}{2},1)$ we have 
	\begin{equation*}
	\log\mathbb{P}\left[n_L\left(r\right)\geq V\right]\leq-\tilde{c}_L\left(1-r\right)V^2.
	\end{equation*}
\end{claim}

In this part we'll use the following parameters:
\begin{equation*}
\delta\coloneqq 1-r \ , R\coloneqq\tfrac{1}{2}\left(1+r\right)=1-\tfrac{1}{2}\delta\ ,\ r_0\coloneqq 1-2\delta\ , \ N\coloneqq\eta\left[V\right],
\end{equation*} 
where $\eta$ is a sufficiently small constant that will be chosen later on.
Recall the Poisson-Jensen formula, if $f\in\text{Hol}\left(r\bar{\mathbb{D}}\right)$, $\mathcal{Z}_f=f^{-1}\{0\}$ then,
\begin{equation*}
\log\left|f\left(z\right)\right|=\int_{0}^{2\pi}\log\left|f\left(re^{i\theta}\right)\right|\frac{r^2-\left|z\right|^2}{\left|re^{i\theta}-z\right|^2}d\mu\left(\theta\right)+\sum_{\lambda\in\mathcal{Z}_f\cap r\bar{\mathbb{D}}}\log\left|\frac{r\left(z-\lambda\right)}{r^2-\bar{\lambda}z}\right|
\end{equation*}
 and also Jensen's formula, for $z=0 , f(0)\neq0$ 
 \begin{equation*}
 \log\left|f\left(0\right)\right|+\int_{0}^{r}\frac{n\left(t\right)}{t}dt=\int_{0}^{2\pi}\log\left|f\left(re^{i\theta}\right)\right|d\mu\left(\theta\right).
 \end{equation*} 
 Hence we have, 
 \begin{equation}
 \label{eq:integral inequality}
 n_L\left(r\right)\log\frac{R}{r}\leq\left(\int_{R\mathbb{T}}-\int_{r\mathbb{T}}\right)\log\left|f\right|d\mu\leq\log M_f\left(R\right)-\int_{r\mathbb{T}}\log|f|d\mu
 \end{equation}
 where $M_f\left(r\right)=\max_{r\mathbb{D}}\left|f\right|$.
 \begin{lemma}[{\cite[Lemma~2]{holedisk}}] 
 	\label{lemma: maximum large}
 		Let $f$ be a GAF on $\mathbb{D}$, and $s\in\left(0,\delta\right)$. Put $$\sigma_f^2(r)=\max_{r\bar{\mathbb{D}}}\mathbb{E}[|f|^2].$$ Then, for every $\lambda>0$, $$\mathbb{P}\big[M_f(r)>\lambda\sigma_f(r+s)\big]\leq \frac{C}{s}e^{-c\lambda^2}.$$
 	\end{lemma}
	For the hyperbolic GAF $f$, we have $\sigma_f\left(t\right)=\left(1-t^2\right)^{-\tfrac{L}{2}}$. Hence 
	\begin{equation*}
	2^{-\tfrac{L}{2}}\left(1-t\right)^{-\tfrac{L}{2}}\leq\sigma_f\left(t\right)\leq\left(1-t\right)^{-\tfrac{L}{2}},
	\end{equation*} 
	whence, 
	\begin{equation}\label{pointwise variance}
	\mathbb{P}\left[M_f\left(t\right)>\lambda\left(1-\left(t+s\right)\right)^{{-\tfrac{L}{2}}}\right]\leq\mathbb{P}\left[M_f\left(t\right)>\lambda\sigma_f\left(t+s\right)\right].
	\end{equation}
	Choose $t=R,s=\tfrac{1}{4}\delta,$ and $\lambda=\tfrac{1}{2}\sqrt{1-r}V$ in \eqref{pointwise variance} and use \Cref{lemma: maximum large} to get  
	\begin{equation}\label{maximum large than}
	\mathbb{P}\left[M_f\left(R\right)>\left(\tfrac{\delta}{4}\right)^{\frac{1}{2}-\frac{L}{2}}V\right]\leq C\delta^{-1}e^{-c\left(1-r\right)V^2}.
	\end{equation}	
	Using that $\delta=1-r$ in \eqref{maximum large than}, we get 
	\begin{equation}
	\mathbb{P}\left[M_f\left(R\right)>2^{L-1}\left(1-r\right)^{\tfrac{1-L}{2}}V\right]\leq C\delta^{-1}e^{-c\left(1-r\right)V^2}.
	\end{equation}
Denote by $E_r^{(1)}:=\{M_f\left(R\right)>2^{L-1}\left(1-r\right)^{\tfrac{1-L}{2}}V\}$. \\ Then, outside $E_r^{(1)}$, the following holds:
{ 	
\begin{equation}\label{maximum not too large}
	\begin{split}
		\log M_f\left(R\right)&\leq \left(L-1\right)\log 2+\log V+\tfrac{L-1}{2}\log\left(1-r\right)^{-1}
		\leq  a_L\log V
		\end{split} 
\end{equation}
where the last inequality follows from our assumption \eqref{assumption}.
Next, set $Z_r=\left\{n_L\left(r\right)\geq V\right\}$; using \eqref{eq:integral inequality} and \eqref{assumption}, we see that on $Z_r\backslash E_r^{(1)}$ we have: 
\begin{equation*}
\int_{r\mathbb{T}}\log\left|f\right|d\mu\leq 
-d_L\left(1-r\right)V,
\end{equation*} where $d_L=-\frac{1}{2}+\frac{a_L}{C_L}$.}

Choosing $N$ points $\{r_0e^{i\theta_j}\}_{j=1}^{N}$ with $r_0<r$,
\begin{align*}
\int_{0}^{2\pi}\frac{1}{N}\sum_{j=1}^{N}\log\left|f\left(r_0e^{i\theta_j}\cdot e^{i\beta}\right)\right|d\mu\left(\beta\right)&=\frac{1}{N}\sum_{j=1}^{N}\int_{0}^{2\pi}\log\left|f\left(r_0e^{i\theta_j}\cdot e^{i\beta}\right)\right|d\mu\left(\beta\right)
\\&
=\log\left|f\left(0\right)\right|+\int_{0}^{r_0}\frac{n_L\left(t\right)}{t}dt.
\end{align*} 
By continuity away from zeros of $\log\left|f\right|$ we can find $\beta^*$ such that
{
\begin{equation}
\label{estimate:avergae}
\frac{1}{N}\sum_{j=1}^{N}\log\left|f\left(r_0e^{i\theta_j}\cdot e^{i\beta^*}\right)\right|=\int_{r_0\mathbb{T}}\log\left|f\right|d\mu\leq\int_{r\mathbb{T}}\log\left|f\right|d\mu\leq-d_L\left(1-r\right)V.
\end{equation}
}
In the next step, we wish to avoid the dependence on $\beta^*$ by expanding this estimate to a small arc on $r_0\mathbb{T}$ centered at $r_0e^{i\beta^*}$. To this end, we will bound the angular derivative of $\log\left|f\right|$ assuming that $f$ doesn't vanishes in a thin ring around the circle $r_0\mathbb{T}$.

Let $\left|z\right|=r_0$ and assume that $f\neq0$ on $r_0-\epsilon\leq\left|z\right|\leq r_0+\epsilon$ with $\epsilon=\epsilon\left(r\right)\ll(1-r)^3$. Differentiating the Poisson-Jensen, we get
\begin{align*}
\frac{\partial\log\left|f\left(r_0e^{i\varphi}\right)\right|}{\partial\varphi}&=\int_{0}^{2\pi}\frac{{2}r_0\cdot r\left(r^2-r_0^2\right){\sin\left(\varphi-\theta\right)}}{\left|re^{i\theta}-r_0e^{i\varphi}\right|^4}\cdot\log\left|f\left(re^{i\theta}\right)\right|d\mu\left(\theta
\right)
\\&
{
+\sum_{\lambda\in\mathcal{Z}_f\cap r\bar{\mathbb{D}}} r_0|\lambda|\sin\left(\theta-\varphi_\lambda\right)\left[\frac{1}{\left|r_0e^{i\varphi}-\lambda\right|^{2}}-\frac{r^2}{\left|r^2-\bar{\lambda}r_0e^{i\varphi}\right|^{2}}\right].}
\\&
\end{align*}
where $\lambda=|\lambda|e^{i\varphi_\lambda}$. 
Then,
\begin{align*}
\left|\frac{\partial\log\left|f\left(r_0e^{i\varphi}\right)\right|}{\partial\varphi}\right|&\leq\frac{{4}}{\left(r-r_0\right)^3}\int_{0}^{2\pi}\left|\log\left|f\left(re^{i\theta}\right)\right|\right|d\mu\left(\theta\right)
\\& +\sum_{\lambda\in\mathcal{Z}_f\cap r\bar{\mathbb{D}}}\left[\frac{1}{\underset{\lambda\in\mathcal{Z}_f\cap r\bar{\mathbb{D}}}{\min}\left|r_0e^{i\varphi}-\lambda\right|^2}+\frac{1}{\left(r-r_0\right)^2}\right]
\\&\leq \frac{{4}}{\left(r-r_0\right)^3}\int_{0}^{2\pi}\left|\log\left|f\left(re^{i\theta}
\right)\right|\right|d\mu\left(\theta\right)+\frac{2n_L\left(r\right)}{\epsilon^2}
\end{align*}

Using \eqref{eq:integral inequality} and \eqref{maximum not too large}, we bound each term on the right hand side.

	 The second summand is bounded by 
	 { 
	\begin{align*}
	\frac{2n_L\left(r\right)}{\epsilon^2}&\leq\frac{2}{\epsilon^2}\left[\log M_f\left(R\right)-\int_{r\mathbb{T}}\log\left|f\right|d\mu\right]
	\lesssim_L\frac{1}{\epsilon^2}\left[\log V+\int_{r\mathbb{T}}\log\_\left|f\right|d\mu\right]
	\\& \overset{\eqref{assumption}}{\lesssim}\frac{1}{\epsilon^2}\left[\left(1-r\right)V +\int_{r\mathbb{T}}\log\_\left|f\right|d\mu\right] .
	\end{align*} 
The first summand is bounded by
	\begin{align*}
	\frac{4}{\left(r-r_0\right)^3}\int_{0}^{2\pi}\left|\log\left|f\left(re^{i\theta}
	\right)\right|\right|d\mu&\lesssim\frac{1}{\left(r-r_0\right)^3}\left[\log M_f\left(r\right)+\int_{r\mathbb{T}}\log\_\left|f\right|d\mu\right]
	\\&
	\lesssim\frac{1}{\left(r-r_0\right)^3}\left[\left(1-r\right)V+\int_{r\mathbb{T}}\log\_\left|f\right|d\mu\right].
	\end{align*}
	Hence, 
	\begin{equation}
	\label{estimate:angular derivative}
	\max_{\varphi}\left|\frac{\partial\log\left|f\left(r_0e^{i\varphi}\right)\right|}{\partial\varphi}\right|\lesssim\frac{1}{\epsilon^2} \left(\left(1-r\right)V+\int_{r\mathbb{T}}\log\_\left|f\right|d\mu\right).
	\end{equation}
}
\begin{lemma}
	\label{Lemma:function not zero}
	Set $\epsilon:=e^{-s\left(1-r\right)V^2}$ with some fixed constant $s>0$. Then, outside a negligible event $E_r^{(2)}$, $f\neq0$ in the annulus $\left\{r_0-\epsilon\leq\left|z\right|\leq r_0+\epsilon\right\}$ . 
\end{lemma}
{
\begin{proof}
	Applying Chebyshev's inequality, we get
	\begin{align*}
	\mathbb{P}\left[n_L\left(r_0+\epsilon\right)-n_L\left(r_0-\epsilon\right)\geq1\right]&\leq\mathbb{E}\left[n_L\left(r_0+\epsilon\right)-n_L\left(r_0-\epsilon\right)\right]
	\\&
	=L\left(\frac{\left(r_0+\epsilon\right)^2}{1-\left(r_0+\epsilon\right)}-\frac{\left(r_0-\epsilon\right)^2}{1-\left(r_0-\epsilon\right)^2}\right)\\&\lesssim_L\epsilon\left(1-r\right)^{-2}. 
		\end{align*}
	Therefore, for every $s'<s$, we get (assuming $r$ is sufficiently close to 1)
	\begin{equation*}
	\mathbb{P}\left[E_r^{(2)}\right]\leq e^{-s'\left(1-r\right)V^2},
	\end{equation*} 
	proving the lemma.
\end{proof}
}
Next, to give an estimate for the integral $\int_{r\mathbb{T}}\log_{-}\left|f\right|d\mu$, we want to make sure that $M_f\left(r_0\right)$ is not too small.
{
\begin{lemma}
	\label{lemma:function not large}
	Let $E_r^{(3)}=\left\{\log M_f\left(r_0\right)\leq-\left(1-r\right)V\right\}$. Then 
	\begin{equation*}
	\mathbb{P}\left[E_r^{\left(3\right)}\right]\leq e^{-a_L\left(1-r\right)V^2}.
	\end{equation*}
\end{lemma}
\begin{proof}
		Assume that $M_f\left(r_0\right)\leq e^{-\left(1-r\right)V}$ , then by Cauchy’s inequalities
		\begin{equation*}
		\left|\xi_n\right|\binom{-L}{n}^{\frac{1}{2}}\leq \frac{M_f\left(r_0\right)}{{r_0}^n}\ ,\ n\geq0.
		\end{equation*}
		  Since for all $n>0$, $\binom{-L}{n}^{\frac{1}{2}}\geq c_Ln^{\frac{1}{2}\left(L-1\right)}$ we get that, 
		  \begin{equation*}
		  \left|\xi_n\right|\leq B_Ln^{\frac{1}{2}\left(1-L\right)}{r_0}^{-n}e^{-\left(1-r\right)V}.
		  \end{equation*}
			Recall that $r_0= 1-2\delta$ so, when $r\geq\frac{1}{2}$, $\delta\leq\frac{1}{2}$, hence $r_0^{-1}\leq e^{2\left(1-r\right)}$ and then,
			\begin{equation*}
			\left|\xi_n\right|\leq B_Ln^{\frac{1}{2}\left(1-L\right)}e^{\left(1-r\right)\left(2n-V\right)}.
			\end{equation*}
	 In the range $1\leq n\leq \frac{1}{4}V$, the right-hand side doesn't exceed (assuming $C_L$ is large)
	 \begin{equation*}
	 e^{-b_L\left(1-r\right)V}
	 \end{equation*}
	 and the probability that $\left|\xi_n\right|$ is smaller than this term doesn't exceed
	 \begin{equation*}
	 e^{-2b_L\left(1-r\right)V}.
	 \end{equation*}
	Overall,
	\begin{equation*}
	\mathbb{P}\left[E_r^{\left(3\right)}\right]\leq \left(e^{-2b_L\left(1-r\right)V}\right)^{\frac{1}{4}V}=e^{-a_L\left(1-r\right)V^2}.
	\end{equation*}
\end{proof}
}
	 We estimate $\int_{r\mathbb{T}}\log\_\left|f\right|$ assuming that we are outside of the event $E_r^{(1)}\cup E_r^{(2)}\cup E_r^{(3)}$.
	 By \Cref{lemma:function not large}, there exists $w\in r_0\mathbb{T}$ such that $\log\left|f\left(w\right)\right|\geq-\left(1-r\right)V$. 
	 Let $P\left(z, w\right)$ be the Poisson kernel for the disk $r{\mathbb{D}}$, evaluated at $w$. Then, using subharmonicity of $\log\left|f\right|$ in $r{\mathbb{D}}$ we have,
	 \begin{equation*}
	 -\left(1-r\right)V\leq\log\left|f\left(w\right)\right|\leq\int_{r\mathbb{T}}\log\left|f\right| P\left(\cdot,w\right)d\mu=\int_{r\mathbb{T}}\left(\log_+\left|f\right|-\log\_\left|f\right|\right) P\left(\cdot,w\right)d\mu,
	 \end{equation*}
	 whence,
	 \begin{align*}
	 \frac{r-r_0}{r+r_0}\int_{r\mathbb{T}}\log\_\left|f\right|d\mu&\leq\int_{r\mathbb{T}}\log\_\left|f\right| P\left(\cdot,w\right)d\mu
	 \\&\leq\int_{r\mathbb{T}}\log_+\left|f\right| P\left(\cdot,w\right)d\mu+\left(1-r\right)V
	\\&
	 \leq\frac{r+r_0}{r-r_0}\int_{r\mathbb{T}}\log_+\left|f\right|d\mu+\left(1-r\right)V\leq \frac{r+r_0}{r-r_0}\log M_f\left(r\right)+\left(1-r\right)V
	\\&
	 \lesssim_L\frac{r+r_0}{r-r_0}\left(1-r\right)V+\left(1-r\right)V.
	 \end{align*}
	 Therefore, assuming $r$ is sufficiently close to 1,
	 \begin{equation*}
	 {\int_{r\mathbb{T}}\log\_\left|f\right|d\mu\lesssim\frac{1}{\left(r-r_0\right)^2}\left(1-r\right)V.}
	 \end{equation*}
	 \\
	 Finally, we can plug this into \eqref{estimate:angular derivative} and obtain, 
	 \begin{equation}
	 \label{esitmate: angular derivative final}
	 \begin{split}
	 	 \max_{\varphi}\left|\frac{\partial\log\left|f\left(r_0e^{i\varphi}\right)\right|}{\partial\varphi}\right|\lesssim\frac{1}{\epsilon^2}\left(1-r\right)V\leq e^{s''\left(1-r\right)V^2} ,&\quad s''>s, \\ \text{where $r$ is sufficiently close to $1$.}
	 \end{split}
	 \end{equation}
	Combining this with \eqref{estimate:avergae} we get that for ${\left|\beta-\beta^*\right|\leq\Delta:=e^{-2s''V^2\left(1-r\right)}}$,
	\begin{align*}
	\frac{1}{N}\sum_{j=1}^{N}\log\left|f\left(r_0e^{i\theta_j}\cdot e^{i\beta}\right)\right|&-\frac{1}{N}\sum_{j=1}^{N}\log\left|f\left(r_0e^{i\theta_j}\cdot e^{i\beta^*}\right)\right|
	\\&
	\leq\left|\beta-\beta^*\right|\cdot\sup_{\varphi}\left|\frac{\partial\log\left|f\left(r_0e^{i\varphi}\right)\right|}{\partial\varphi}\right|,
	\end{align*}
	which implies
	\begin{equation}
	{
	\label{estimae:averge final}
	\frac{1}{N}\sum_{j=1}^{N}\log\left|f\left(r_0e^{i\theta_j}\cdot e^{i\beta}\right)\right|\leq -\tfrac{d_L}{2}\left(1-r\right)V.}
	\end{equation}

	Now we are ready to complete the proof. Consider the product space $\left(\mathbb{T},m\right)\times\left(\Omega,\mathbb{P}\right)$ with the probability measure $\mathbb{Q}=m\times\mathbb{P}$. Here $\beta$ is chosen uniformly in $[0,2\pi]$ and $\Omega$ is the probability space for our GAF $f$. 
	Consider the following events: 
	\begin{itemize}[label={--}]
		{
		\item $H=\left\{\left(\beta,\omega\right) \mid n_L\left(r\right)\geq V\right\}$ $=\mathbb{T}\times Z_r$
		\item $K=\left\{\left(\beta,\omega\right)\mid\max_{\varphi}\left|\frac{\partial\log\left|f\left(r_0e^{i\varphi}\right)\right|}{\partial\varphi}\right|\leq e^{s''\left(1-r\right)V^2}\right\}\supset\mathbb{T}\times\left(Z_r\backslash E_r^{(1)}\cup E_r^{(2)} \cup E_r^{(3)}\right)$
		\item $J=\left\{\left(\beta,\omega\right) \mid f\neq0 \ \text{in} \ \text{ann}\left(0,r_0-e^{-s\left(1-r\right)V^2},r_0+e^{-s\left(1-r\right)V^2}\right)\right\}=\mathbb{T}\times\left(\Omega\backslash E_r^{(3)}\right)$
		\item $C=\left\{\left(\beta,\omega\right) \mid \frac{1}{N}\sum_{j=1}^{N}\log|f(r_0e^{i\theta_j}\cdot e^{i\beta})|\leq-d_L\left(1-r\right)V\right\}$
		\item $D=\{(\beta,\omega) \mid |\beta-\beta^*(\omega)|\leq\Delta=e^{-2s''\left(1-r\right)V^2}\}$}
	\end{itemize}

\begin{remark}
		\rm{
		Recall that $\beta^*:\Omega\mapsto[0,2\pi)$, where $\beta^{*}$ was defined in \eqref{estimate:avergae}.
		We mention that $\beta^*$ may indeed be chosen to be measurable with respect to $\Omega$. See \Cref{measurabiliy} for a more detailed explanation.}    
	\end{remark}
	
Next, note that $H,K,J$ do not depend on $\beta$, and that $\mathbbm{1}_C\left(\beta,\omega\right)\geq\mathbbm{1}_{H\cap J \cap K\cap D}\left(\beta,\omega\right)$. \newline
Also, the events $H\cap J\cap K$ and $D$ are independent. Indeed, 
\begin{equation*}
\mathbbm{1}_{H\cap J \cap K\cap D}(\beta,\omega)=\mathbbm{1}_{H\cap J\cap K}(\beta,\omega)\cdot\mathbbm{1}_D(\beta,\omega)=\mathbbm{1}_{H\cap J\cap K}(0,\omega)\cdot\mathbbm{1}_D(\beta,\omega),
\end{equation*}
and by Fubini's theorem,
\begin{align*}
\mathbb{Q}\left(H\cap J\cap K\cap D \right)&=\int_{\mathbb{T}\times\Omega}\mathbbm{1}_D\left(\beta,\omega\right)\mathbbm{1}_{H\cap J\cap K}(0,\omega)d\mathbb{Q}
\\&
=\int_{\Omega}\left[\int_{\left\{\left|\beta-\beta^*\left(\omega\right)\right|\right\}}dm\left(\beta\right)\right]\mathbbm{1}_{H\cap J\cap K}\left(0,\omega\right)d\mathbb{P}\left(\omega\right)
\\&
=2\Delta\int_{\Omega}\mathbbm{1}_{H\cap J\cap K}\left(0,\omega\right)d\mathbb{P}\left(\omega\right)=\mathbb{Q}\left(D\right)\cdot \mathbb{Q}\left(H\cap J\cap K\right)
\end{align*}
as desired.
Then,  $\mathbbm{1}_C\left(\beta,\omega\right)\geq\mathbbm{1}_{H\cap J\cap K}\left(0,\omega\right)\cdot\mathbbm{1}_{D}\left(\beta,\omega\right)$
and
\begin{align*}
\mathbb{Q}\left(C\right)&\geq2\Delta P\left(H \cap J \cap K\right)\geq2\Delta\left[\mathbb{Q}\left(H\right)-\mathbb{Q}\left(K^{\complement}\right)-\mathbb{Q}\left(J^{\complement}\right)\right]
\\&
\geq \Delta\left[\mathbb{P}\left(Z_r\right)-\mathbb{P}\left(E_r^{(1)}\right)-\mathbb{P}\left(E_r^{(2)}\right)-\mathbb{P}\left(E_r^{(3)}\right)\right] 
\end{align*}
whence 
\begin{equation*}
\label{estimate:Large Deviation}
\mathbb{P}\left(Z_r\right)\leq\frac{1}{\Delta}\mathbb{Q}\left(C\right)+\mathbb{P}\left(E_r^{(1)}\right)+\mathbb{P}\left(E_r^{(2)}\right)+\mathbb{P}\left(E_r^{(3)}\right) 
\end{equation*}
\begin{equation}\label{finish probabilities}
{\leq e^{2s''\left(1-r\right)V^2}\mathbb{Q}\left(C\right)+e^{-a_L\left(1-r\right)V^2}.}
\end{equation}
The following lemma is required to finish the proof of the upper bound.
\begin{lemma}
	\label{lemma:avreage with beta}
	For every fixed $\beta$, 
	\begin{equation}
	\mathbb{P}\left[\frac{1}{N}\sum_{j=1}^{N}\log\left|f\left(r_0e^{i\theta_j}\cdot e^{i\beta}\right)\right|\leq-\tfrac{d_L}{2}\left(1-r\right)V\right]\leq e^{-c_L\left(1-r\right)V^2}.
	\end{equation}
\end{lemma}
Using this lemma and Fubini's theorem we get,
\begin{equation*}
\mathbb{Q}(C)=\int\left[\int \mathbbm{1}_C(\beta,\omega)d\mathbb{P}(\omega)\right]dm\left(\beta\right)\leq e^{-c_L \left(1-r\right)V^2}.
\end{equation*}
Plugging in $\Delta$ together with this bound into  \Cref{finish probabilities} we get, 
\begin{equation*}
\mathbb{P}\left(Z_r\right)\leq e^{\left(2s''-c_L\right)\left(1-r\right)V^2}+e^{-a_LV^2\left(1-r\right)}.
\end{equation*}
Finally we choose $s''<\frac{1}{2}c_L$ and get the desired upper-bound.

\begin{proof}[Proof of \Cref{lemma:avreage with beta}]
Fix $\beta$,$\zeta>0$ by Chebyshev’s inequality,
\begin{align*}
&\mathbb{P}\left[\sum_{j=1}^{N}\log\left|f\left(r_0e^{i\theta_j}\cdot e^{i\beta}\right)\right|\leq -\tfrac{d_L}{2}N\left(1-r\right)V\right]\\&= \mathbb{P}\left[\prod_{j=1}^{N}\left|f\left(r_0e^{i\theta_j}\cdot e^{i\beta}\right)\right|^{-\zeta}\geq e^{\tfrac{d_L}{2}\zeta N\left(1-r\right)V}\right]
\\&
\leq e^{-\tfrac{d_L}{2}\zeta N\left(1-r\right)V}\cdot\mathbb{E}\big[\prod_{j=1}^{N}|f(r_0e^{i\theta_j}\cdot e^{i\beta})|^{-\zeta}\big].
\end{align*}
To estimate the expectation on the right-hand side, we use several lemmas from \cite{holedisk}. 
\begin{lemma}[{\cite[Lemma~15]{holedisk}}] Suppose that $\left(\eta_j\right)_{1\leq j\leq N}$ are complex Gaussian random variables with covariance matrix $\Sigma$. Then, for $0\leq\zeta<2$,
\begin{equation*}
\mathbb{E}\left[\prod_{j=1}^{N}\left|\eta_j\right|^{-\zeta}\right]\leq\frac{1}{\det\Sigma}\left(\Lambda^{\left(1-\frac{1}{2}\zeta\right)}\cdot\Gamma\left(1-\frac{1}{2}\zeta\right)\right)^N,
\end{equation*}	
where $\Lambda$ is the largest eigenvalue of $\Sigma$. 	
\end{lemma}
\begin{lemma}[{\cite[Lemma~10]{holedisk}}]
	Let $F$ be any Gaussian Taylor series with radius of convergence $R$ and let $z_j=re\left(j/N\right)$ for $r<R$ and $j=0,\dots,N-1$. Consider the covariance matrix $\Sigma=\Sigma\left(r,N\right)$ of the random variables $F\left(z_0\right),\dots,F\left(z_{N-1}\right)$, that is,
	\begin{equation*}
	\Sigma_{jk}=\mathbb{E}\left[F\left(z_j\right)\overline{F\left(z_k\right)}\right]=\sum_{n\geq0}a_n^2r^{2n}e\left(\left(j-k\right)n/N\right).
	\end{equation*} 
	Then, the eigenvalues of $\Sigma$ are $$\lambda_m=N\cdot\sum_{n\equiv m(N)}a_n^2r^{2n}, \ \ m=0,\dots,N-1,$$ where $n\equiv m(N)$ denotes that $n$ is equivalent to $m$ modulo $N$. 
\end{lemma}

\begin{lemma}\label{lemma from BNPS}{\cite[Sec.~6.2]{holedisk}}\ Let $\Sigma=\Sigma\left(r,N\right)$ be the covariance matrix of the Gaussian random variables $F\left(z_0\right),\dots,F\left(z_{N-1}\right)$, $z_j=re\left(j/N\right)$, and let $\Lambda=\Lambda\left(r,N\right)$ be the largest eigen value of $\Sigma$. Then, 
 \begin{equation*}
		\Lambda\leq\begin{cases}
		\frac{C\cdot N}{\delta^{L-1}}, & L>1\\
		C\cdot N, & L\leq1\\
		\end{cases} \quad \text{and} \quad\det\Sigma\geq e^{\left(1+o\left(1\right)\right)\left(LN\log N-2\delta N^2\right)}.
		\end{equation*} 

\end{lemma}

\begin{proof}[Proof of \Cref{lemma from BNPS}]	
		We begin with a lower bound for $\text{det}\Sigma$. In our case $a_m^2=\binom{-L}{m}\geq c\left(m+1\right)^{L-1}$, hence 
		\begin{align*}
		\text{det}\Sigma&=\prod_{m=0}^{N-1}\lambda_m\left(\Sigma\right)\geq\prod_{m=0}^{N-1}Na_m^2\left(1-2\delta\right)^{2m}
		\geq N^N\left(1-2\delta\right)^{N(N-1)}\prod_{m=0}^{N-1}c\left(m+1\right)^{L-1}
		\\&
		=\left(cN\right)^N\left(N!\right)^{L-1}\left(1-2\delta\right)^{N\left(N-1\right)}
		=\exp\left(-4\delta N^2+LN\log N+O\left(N\right)\right)
		\end{align*}
		For the upper bound for $\Lambda$,
		assume $L\leq1$ (the case $L>1$ appears in \cite{holedisk}).
		Since $$a_{m+jN}^2=\binom{-L}{m+jN}\sim_L \left(m+jN\right)^{L-1} \ \text{as} \ n\rightarrow\infty,$$ we  have for small $\delta$ and $j\geq0$,
		\begin{align*}
		\frac{a_{m+\left(j+1\right)N}^2r_0^{2\left(m+\left(j+1\right)N\right)}}{a_{m+jN}^2r_0^{2\left(m+jN\right)}}&\leq D_L\left(\frac{m+(j+1)N}{m+jN}\right)^{L-1} r_0^{2N}
		\\&
		=D_L\left(1+\frac{N}{m+jN}\right)^{L-1} r_0^{2N}
		\\&\leq e^{-4\delta N}.
		\end{align*}
		Take $r$ close enough to $1$, such that   
		\begin{equation*}
		\frac{a_{m+(j+1)N}^2r_0^{2(m+(j+1)N)}}{a_{m+jN}^2r_0^{2(m+jN)}}\leq\frac{1}{2}.
		\end{equation*}
		So,
		\begin{equation*}
		\lambda_m\left(\Sigma\right)\leq C\cdot N\cdot a_m^2r_0^{2m}\sum_{j\geq0}2^{-j}=C\cdot N\cdot a_m^2r_0^{2m}\leq C\cdot N.
		\end{equation*}
		(where the last inequality holds because $a_m^2\leq1$).
		\end{proof}
	Now we readily complete the proof of \Cref{lemma:avreage with beta},and therefore, of \Cref{Theorem: overcrowding}
Set $\zeta:=2-\delta$. Then, using $\Gamma\left(z\right)\sim 
(\frac{1}{z})$ as $z\rightarrow 0$, $\Gamma\left(1-\frac{1}{2}\zeta\right)=\Gamma\left(\frac{\delta}{2}\right)\leq\frac{2}{\delta}=\frac{2}{1-r}$
we get, 
\begin{align*}
\mathbb{E}\big[\prod_{j=1}^{N}|f(r_0e^{i\theta_j}\cdot e^{i\beta})|^{-\zeta}\big]&\leq e^{4\delta N^2-LN\log N+O\left(N\right)}\left(\Lambda^{\frac{\delta}{2}}\cdot\frac{1}{\delta}\right)^N.
\end{align*}
Hence,

\begin{equation}\label{bound expectation}
\begin{split}
 \log\mathbb{E}\left[\prod_{j=1}^{N}|f(r_0e^{i\theta_j}\cdot e^{i\beta})|^{-\zeta}\right]&\leq4\delta N^2-LN\log N+\delta N\log\Lambda+N\log\frac{1}{1-r}+O\left(N\right)
\\&
\leq 4\delta N^2+B_L\delta N\log(N)+N\log\frac{1}{1-r}+O\left(N\right).
\end{split}
\end{equation}
Recall that $N=\eta \left[V\right]$ with a small constant $\eta$ and that $V$
satisfies assumption \eqref{assumption}.
Then the expression on the RHS of \eqref{bound expectation} is bounded by 
\begin{equation*}
\left(4\eta^2+\frac{\max\{L+1,2\}}{C_L}\eta\right)\left(1-r\right)V^2
\end{equation*}
assuming $r$ is sufficiently close to $1$. 
Also, recall that $$d_L=-\frac{1}{2}+\frac{a_L}{C_L}$$ with fixed $a_L$. Therefore, since $\zeta\geq1$ and $C_L$ is large, we may choose $\eta$ sufficiently small and get 
\begin{align*}
\log\mathbb{P}\left[\sum_{j=1}^{N}\log\left|f\left(r_0e^{i\theta_j}\cdot e^{i\beta}\right)\right|\leq
 -\tfrac{d_L}{2}N\left(1-r\right)V\right]\leq-\tilde{c}_L\left(1-r\right)V^2.
\end{align*}
\end{proof}

\appendix

\section{The Case $L=1$ - proof of \Cref{Theorem: babycase}}
\label{proofbabycase}
Recall that for $L=1$, the random variable $n_1\left(r\right)$ is distributed as a sum of independent Bernoulli random variables (see\cite{Determinantal}),
\begin{equation}
\begin{split}
\mu_r & =\mathbb{E}\left[\sum_{k=1}^{\infty}X_k\right]=\sum_{k=1}^{\infty}r^{2k}=\frac{r^2}{1-r^2}\ \underset{r\uparrow 1}{\sim} \ \frac{1}{2}\cdot\frac{1}{1-r},\\
v_1\left(r\right) & =v_1\left(r\right)=\Var\left[\sum_{k=1}^{\infty}X_k\right]=\sum_{k=1}^{\infty}r^{2k}\left(1-r^{2k}\right)=\frac{r^2}{1-r^4} \ \underset{r\uparrow 1}{\sim} \ \frac{1}{4}\cdot\frac{1}{1-r}.
\end{split}
\end{equation}
Here, it is not hard to calculate the moment generating function, hence we can use large-deviations techniques to find the \emph{rate function} to be defined in \eqref{rate function} and get the rate of decay of the large deviations probability (see \cite[Section~2.3]{Deviations_book} )

Now, in order to prove \Cref{Theorem: babycase}, we set: \begin{equation*}
Z_r:=\frac{1}{v_1\left(r\right)^{\alpha}}\sum_{k=1}^{\infty}(X_k-r^{2k})\ ,$$
$$\epsilon_r:=\frac{1}{v_1\left(r\right)^{2\alpha-1}} \ ,
\end{equation*} where $r\in(0,1)$. Let $\Lambda_r\left(\cdot\right)$ be the logarithmic moment generating function of $Z_r$ i.e. $\Lambda_r(\lambda)=\log\mathbb{E}\left[e^{\lambda Z_r}\right]$ and 
\begin{equation}\label{limiting}
\Lambda\left(\lambda\right):= \underset{r\uparrow1}{\lim} \epsilon_r\Lambda_r\left(\epsilon_r^{-1}\lambda\right).
\end{equation}
The rate function is defined as follows
\begin{equation}\label{rate function}
\Lambda^*\left(x\right)=\underset{\lambda\in\mathbb{R}}{\sup}\left\{\lambda x-\Lambda\left(\lambda\right)\right\}
\end{equation}

\subsection{Calculation of $\Lambda\left(\lambda\right)$}
\begin{lemma}\label{Lemma:Calculation moment genertaing}
	For $\alpha>\frac{1}{2}$, the limit \eqref{limiting} exists, and differentiable for every $\lambda\in\mathbb{R}$.
	\begin{enumerate}
		\item For $\frac{1}{2}<\alpha<1$,
		\begin{equation*}
		\Lambda\left(\lambda\right) =\frac{\lambda^2}{2}
		\end{equation*}
		for every $\lambda \in \mathbb{R}.$
		\item For $\alpha=1$ 
		\begin{equation*}
		\Lambda\left(\lambda\right)=-2\lambda-2\Li_2\left(1-e^\lambda\right)
		\end{equation*}
		for every $\lambda\in\mathbb{R}$.
		\item For $\alpha>1$
		\begin{equation*}
		\Lambda(\lambda)=
		\begin{cases} 
		\lambda^2 & \lambda>0, \\
		0 & \lambda\leq0 .\\
		\end{cases}
		\end{equation*}
	\end{enumerate}
\end{lemma}
In order to apply the \emph{G\"artner-Ellis} theorem \cite[Section~2.3]{Deviations_book}, let us find the rate function for each case. 
\subsection{The \emph{Fenchel-Legendre transform} $\Lambda^{*}\left(x\right)$} 
Using \Cref{Lemma:Calculation moment genertaing}, it is easy to compute $\Lambda^{*}\left(x\right)$ in case $\frac{1}{2}<\alpha<1$ or $\alpha>1$, while the case $\alpha=1$ (which is more complicated) follows from \Cref{lemma: alpha=1} below.
\begin{itemize}
	\item For $\frac{1}{2}<\alpha<1$, \begin{equation}
	\Lambda^{*}\left(x\right)=\frac{x^2}{2}.
	\end{equation}  
	\item For $\alpha=1$ (\Cref{lemma: alpha=1}),
	\begin{equation}
	\label{Alpha=1}
	\Lambda^*\left(x\right)=\begin{cases}
	+\infty & x<-2,\\
	\frac{\pi^2}{3} & x=-2,\\
	h\left(\lambda_{-1}\left(x\right),x\right) & -2<x<0,\\
	h\left(\lambda_0\left(x\right),x\right) & x\geq0,\\
	\end{cases}
	\end{equation} where
	\begin{equation*}
	h\left(y,x\right)=y(x+2)+2\Li\left(1-e^y\right),
	\end{equation*}
	\begin{equation*}
	\lambda_{j}\left(x\right)=\frac{x+2}{2}+W_{j}\left(-\frac{x+2}{2}e^{-\frac{x+2}{2}}\right), \quad j\in\{-1,0\}.
	\end{equation*} $W_{j}$, $j\in\{-1,0\}$ is the principal$/$lower branches of Lambert $W$ function respectively.
	\begin{figure}[H]
		\centering
		\includegraphics{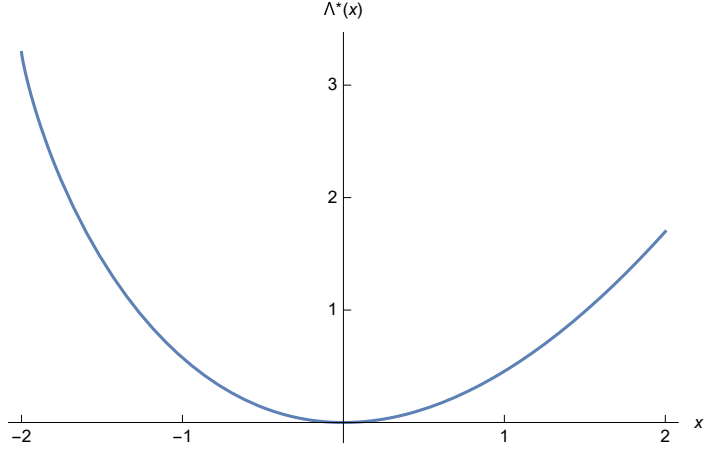}
		\caption{The rate function $\Lambda^*$ for $\alpha=1$}
		\label{fig: rate}
	\end{figure}
	Note that because $\Lambda^*$ is not symmetric the probability of the event $\left|n_1\left(r\right)-\mu_r\right|\geq v_1\left(r\right)$ is dominated by the event $n_1\left(r\right)\geq\mu_r+tv_1\left(r\right)$. Furthermore, the lower bound for the special case of the hole event when $t=-2, \alpha=1$ does not follow from the theory of large deviations, but can be computed directly (see \cite[Section~5.1]{gaf_book}). 
	\item For $\alpha>1$, 
	\begin{equation}
	\Lambda^{*}\left(x\right)=\begin{cases}
	\frac{x^2}{4} & x\geq0,\\
	+\infty & x<0.\\
	\end{cases}
	\end{equation}  
\end{itemize}
\subsection{Proof of \Cref{Theorem: babycase}}
Next, assuming \Cref{Lemma:Calculation moment genertaing}, \Cref{Theorem: babycase} follows from the Gartner-Ellis theorem:
\begin{proof} Notice that
	\begin{align*}
	\mathbb{P}\left[\left|\sum_{k=1}^{\infty}X_k-\mu_r\right|\geq t\cdot v_1\left(r\right)^{\alpha}\right]&=\mathbb{P}\left[\left|\sum_{k=1}^{\infty}\left(X_k-r^{2k}\right)\right|\geq t\cdot v_1\left(r\right)^{\alpha}\right]
	\\ &=\mathbb{P}\left[\left|Z_r\right|\geq t\right]=\mathbb{P}\left[Z_r\in\left(-\infty,-t\right]\cup\left[t,\infty\right)\right]
	\end{align*}
	
	\begin{itemize}
		\item For $\frac{1}{2}<\alpha<1$, it holds that
		\begin{equation*}
		\underset{x\in\left(-\infty,-t\right)\cup\left(t,\infty\right)}{\inf}\Lambda^*\left(x\right)=\underset{x\in\left(-\infty,-t\right]\cup\left[t,\infty\right)}{\inf}\tfrac{1}{2}x^2=\frac{t^2}{2}
		\end{equation*} thus,
		\begin{equation*}
		\underset{r\uparrow1}{\lim}\epsilon_r\log\mu_r((-\infty,-t]\cup[t,\infty))=-\frac{t^2}{2},
		\end{equation*}
		whence,
		\begin{equation*}
		-\log\mathbb{P}\left[\left|\sum_{k=1}^{\infty}X_k-\mu_r\right|\geq v_1\left(r\right)^{\alpha}\right]\sim {\frac{t^2}{2}v_1\left(r\right)^{2\alpha-1}} ,\quad r\uparrow1.
		\end{equation*}  
		\item For $\alpha=1$, it holds that
		\begin{equation*}
		\underset{x\in\left(-\infty,-t\right)\cup\left(t,\infty\right)}{\inf}\Lambda^*\left(x\right)=\underset{x\in\left(-\infty,-t\right]\cup\left[t,\infty\right)}{\inf}\Lambda^*\left(x\right)=\min\left\{\Lambda^*\left(-t\right),\Lambda^*\left(t\right)\right\}=\Lambda^*\left(t\right)
		\end{equation*}
		where the last equality also follows from \Cref{lemma: alpha=1}.
		Thus,
		\begin{equation*}
		\underset{r\uparrow1}{\lim}\epsilon_r\log\mu_r\left(\left(-\infty,-t\right]\cup\left[t,\infty\right)\right)=-\Lambda^*\left(t\right)=-c(t),
		\end{equation*}
		whence,
		\begin{equation*}
		-\log\mathbb{P}\left[\left|\sum_{k=1}^{\infty}X_k-\mu_r\right|\geq t\cdot v_1\left(r\right)^{\alpha}\right]\sim {\frac{c(t)}{4}\cdot\left(\frac{1}{1-r}\right)^{2\alpha-1}} ,\quad r\uparrow1.
		\end{equation*}
		\item For $\alpha>1$, it holds that $\underset{x\in\left(-\infty,-t\right)\cup\left(t,\infty\right)}{\inf}\tfrac{1}{4}x^2=\frac{t^2}{4}$
		thus,
		\begin{equation*}
		\underset{r\uparrow1}{\lim}\epsilon_r\log\mu_r([t,\infty))=-\frac{t^2}{4}, \end{equation*}
		whence,
		\begin{equation*}
		-\log\mathbb{P}\left[\left|\sum_{k=1}^{\infty}X_k-\mu_r\right|\geq t\cdot v_1\left(r\right)^{\alpha}\right]\sim {\frac{t^2}{4}v_1\left(r\right)^{2\alpha-1}} ,\quad r\uparrow1.
		\end{equation*} 
	\end{itemize}
\end{proof}

Finally, we turn back to the calculation of the logarithmic moment generating function:
\begin{proof}[Proof of \Cref{Lemma:Calculation moment genertaing}]
	Let $\lambda\in\mathbb{R}$,
	\begin{align*}
	\Lambda_r\left(\epsilon_r^{-1}\lambda\right)&=\Lambda_r\left(v_1\left(r\right)^{2\alpha-1}\lambda\right)=\log\mathbb{E}\left[e^{v_1\left(r\right)^{2\alpha-1}\lambda Z_r}\right]
	\\&
	=\log\mathbb{E}\left[e^{v_1\left(r\right)^{2\alpha-1}\lambda \frac{1}{v_1\left(r\right)^{\alpha}}\sum_{k=1}^{\infty}\left(X_k-r^{2k}\right)}\right]
	\\&=\log\mathbb{E}\left[e^{v_1\left(r\right)^{\alpha-1}\lambda \sum_{k=1}^{\infty}\left(X_k-r^{2k}\right)}\right]=\sum_{k=1}^{\infty}\log\mathbb{E}\left[e^{v_1\left(r\right)^{\alpha-1}\lambda\left(X_k-r^{2k}\right)}\right]
	\\&
	=\sum_{k=1}^{\infty}\log\left(r^{2k} e^{v_1\left(r\right)^{\alpha-1}\lambda\left(1-r^{2k}\right)}+\left(1-r^{2k}\right) e^{v_1\left(r\right)^{\alpha-1}\lambda\left(-r^{2k}
		\right)}\right)
	\\&
	=\sum_{k=1}^{\infty}\log\left(e^{v_1\left(r\right)^{\alpha-1}\lambda\left(-r^{2k}\right)}\left(r^{2k}e^{v_1\left(r\right)^{\alpha-1}\lambda}+1-r^{2k}\right)\right)
	\\&
	=\sum_{k=1}^{\infty}v_1\left(r\right)^{\alpha-1}\lambda\left(-r^{2k}\right)+\sum_{k=1}^{\infty}\log\left(1+r^{2k}\left(e^{v_1\left(r\right)^{\alpha-1}\lambda}-1\right)\right)
	\\&
	=-\mu_rv_1\left(r\right)^{\alpha-1}\lambda+\sum_{k=1}^{\infty}\log\left(1+r^{2k}\left(e^{v_1\left(r\right)^{\alpha-1}\lambda}-1\right)\right)=(\star).
	\end{align*}
	
	Put $T:=e^{v_1\left(r\right)^{\alpha-1}\lambda}-1$ and note that $-1<T<+\infty$ , $\text{sgn}\left(T\right)=\text{sgn}\left(\lambda\right)=\text{sgn}\left(1+Tr^{2k}\right)$. We look at the map $k\rightarrow\log\left(1+Tr^{2k}\right)$ on $\left[0,\infty\right)$ and see that:
	\begin{itemize}[label={--}]
		\item For $\lambda>0$ the mapping decreases and its range is between $\log\left(1+T\right)=v_1\left(r\right)^{\alpha-1}\lambda$ to $0$.
		\item For $\lambda<0$ the mapping increases and its range is between $v_1\left(r\right)^{\alpha-1}\lambda$ to $0$.  
	\end{itemize} 
	Hence, we can write
	\begin{equation*}
	\sum_{k=1}^{\infty}\log\left(1+Tr^{2k}\right)=\int_{0}^{\infty}\log\left(1+Tr^{2x}\right)dx+O\left(v_1\left(r\right)^{\alpha-1}\right),
	\end{equation*} 
	and using the change of variables
	$u=-Tr^{2x},\ dx=\frac{1}{2\log r}\frac{du}{u}$ we get
	\begin{align*}
	\sum_{k=1}^{\infty}\log\left(1+Tr^{2k}\right)&=-\frac{1}{2\log r}\int_{0}^{-T}\frac{\log\left(1-u\right)}{u}du+O\left(v_1\left(r\right)^{\alpha-1}\right) \\
	&=\frac{1}{2\log r}\Li_2\left(-T\right)+O\left(v_1\left(r\right)^{\alpha-1}\right).
	\end{align*}
	Hence we get,
	\begin{equation*}
	(\star)=-\mu_rv_1\left(r\right)^{\alpha-1}\lambda+\frac{1}{2\log r}\Li_2\left(-T\right)+O\left(v_1\left(r\right)^{\alpha-1}\right),
	\end{equation*} 
	therefore,
	\begin{equation*} \epsilon_r\Lambda_r\left(\epsilon_r^{-1}\lambda\right)=-\mu_rv_1\left(r\right)^{-\alpha}\lambda+\frac{v_1\left(r\right)^{-2\alpha+1}}{2\log r}\Li_2\left(-T\right)+o(1).
	\end{equation*}
	Finally, all that's left is to calculate
	\begin{equation*}
	\lim_{r\uparrow1}\left(-\mu_rv_1\left(r\right)^{-\alpha}\lambda+\frac{v_1\left(r\right)^{-2\alpha+1}}{2\log r}\Li_2\left(-T\right)\right).
	\end{equation*} 
	We split the calculation into three cases based on the values of $\alpha$ and use the asymptotic properties of the $\text{Li}_2$ function that we noted in \Cref{special functions}:
	\begin{enumerate}
		\item For $\frac{1}{2}<\alpha<1$, $T\rightarrow0$ when $r\uparrow1$, hence $T\sim v_1\left(r\right)^{\alpha-1}\lambda$. Thus in this case 
		\begin{equation*}
		\Li_2\left(-T\right)\sim v_1\left(r\right)^{\alpha-1}\lambda+\frac{1}{4}v_1\left(r\right)^{2\alpha-2}\lambda^2,
		\end{equation*}
		therefore,
		\begin{align*}
		&-\mu_rv_1\left(r\right)^{-\alpha}\lambda+\frac{v_1\left(r\right)^{-2\alpha+1}}{2\log r}\Li_2\left(-T\right)
		\\&\sim-\mu_rv_1\left(r\right)^{-\alpha}\lambda-\frac{v_1\left(r\right)^{-\alpha}}{2\log r}\lambda-\frac{v_1\left(r\right)^{-2\alpha+1}}{2\log r}\cdot \frac{v_1\left(r\right)^{2\alpha-2}\lambda^2}{4}
		\\&
		\sim-\lambda v_1\left(r\right)^{-\alpha}\left(\mu_r+\frac{1}{2\log r}\right)+\frac{\lambda^2}{2}
		\sim\frac{\lambda^2}{2} ,\quad \text{as} \ r\uparrow 1.
		\end{align*}
		\item For $\alpha=1$ it follows immediately that  $\Lambda\left(\lambda\right)=-2\lambda-2\Li_2\left(1-e^\lambda\right)$.
		\item For $\alpha>1$, 
		\begin{itemize}
			\item If $\lambda<0$, $T\rightarrow-1$ when $r\uparrow1$, hence $\Li_2\left(-T\right)\rightarrow \Li_2\left(1\right)=\frac{\pi^2}{6}$. Thus
			\begin{equation*}
			\lim_{r\uparrow 1}-\mu_rv_1\left(r\right)^{-\alpha}\lambda+\frac{v_1\left(r\right)^{-2\alpha+1}}{2\log r}\Li_2\left(-T\right)
			=\lim_{r\uparrow 1}-2v_1\left(r\right)^{1-\alpha}-2v_1\left(r\right)^{2-2\alpha}\frac{\pi^2}{6}=0.
			\end{equation*}
			\item If $\lambda>0$, $T\rightarrow+\infty$ when $r\uparrow1$ hence $T\sim e^{v_1\left(r\right)^{\alpha-1}}$. So in this case
			\begin{equation*}
			\Li_2\left(z\right)\sim -\frac{1}{2}\log^2\left(T\right)\sim -\frac{1}{2}v_1\left(r\right)^{2\alpha-2}\lambda^2,
			\end{equation*}
			thus
			\begin{align*}
			&\lim_{r\uparrow 1}-\mu_rv_1\left(r\right)^{-\alpha}\lambda+\frac{v_1\left(r\right)^{-2\alpha+1}}{2\log r}\Li_2\left(-T\right)&
			\\  &=\lim_{r\uparrow 1}-2v_1\left(r\right)^{1-\alpha}\lambda-\left(2v_1\left(r\right)^{2-2\alpha}\right)\left(-\frac{1}{2}v_1\left(r\right)^{2\alpha-2}\lambda^2\right)
			=\lambda^2.
			\end{align*}
			\item For $\lambda=0$, $\epsilon_r\Lambda_r\left(\epsilon_r^{-1}\lambda\right)=0$,
		\end{itemize}
	\end{enumerate}
	proving the claim.
\end{proof}
During the proof we used the formula for $\Lambda^{*}$ in the case $\alpha=1$, we calculate it as follows: 
\begin{lemma}
	\label{lemma: alpha=1}
	\begin{equation*}
	\Lambda^*\left(x\right)=\begin{cases}
	+\infty & x<-2,\\
	\frac{\pi^2}{3} & x=-2,\\
	h\left(\lambda_{-1}\left(x\right),x\right) & -2<x<0,\\
	h\left(\lambda_0\left(x\right),x\right) & x\geq0,\\
	\end{cases}
	\end{equation*}
	where
	\begin{equation*}
	h\left(y,x\right)=y(x+2)+2\Li\left(1-e^y\right),
	\end{equation*}
	\begin{equation*}
	\lambda_{j}\left(x\right)=\frac{x+2}{2}+W_{j}\left(-\frac{x+2}{2}e^{-\frac{x+2}{2}}\right), \quad j\in\{-1,0\}.
	\end{equation*} $W_{j}$, $j\in\{-1,0\}$ is the principal$/$lower branches of Lambert $W$ function respectively. 
\end{lemma}
\begin{proof}
	Define $h\left(\lambda\right)=\lambda\left(x+2\right)+2\Li_2\left(1-e^\lambda\right).$ It holds that
	\begin{equation*}
	\lim_{\lambda\rightarrow+\infty}h\left(\lambda\right)=\lim_{\lambda\rightarrow+\infty}\lambda\left(x+2\right)-\frac{\pi^2}{3}-\log^2\left(e^\lambda-1\right)=-\infty,
	\end{equation*} 
	\begin{equation*}
	\lim_{\lambda\rightarrow-\infty}h\left(\lambda\right)=\lim_{\lambda\rightarrow-\infty}\lambda\left(x+2\right)+2\Li_2\left(1\right)=
	\begin{cases}
	2\Li_2\left(1\right)=\frac{\pi^2}{3} & x=-2,\\
	+\infty & x<-2,\\
	-\infty & x>-2.\\
	\end{cases}
	\end{equation*}
	Consider the equation,
	\begin{equation}
	ye^y=-\frac{1}{2}{\left(x+2\right)}e^{-\tfrac{1}{2}{\left(x+2\right)}}, y=\lambda-\frac{1}{2}{\left(x+2\right)}.
	\end{equation}
	Hence,
	\begin{equation*}\label{star}
	y=W\left(-\frac{1}{2}\left(x+2\right)e^{-\tfrac{1}{2}{\left(x+2\right)}}\right)\implies \lambda=\frac{1}{2}\left(x+2\right)+W\left(-\frac{1}{2}{\left(x+2\right)}e^{-\tfrac{1}{2}{\left(x+2\right)}}\right).
	\end{equation*}
	
	$W_0\left(0\right)=0, W_0\left(-\frac{1}{e}\right)=-1$ and $W_{-1}$ decreases from $W_{-1}(-\frac{1}{e})=-1$ to \mbox{$W_{-1}\left(0^{-}\right)=-\infty$.} 
	Furthermore, it holds that:
	\begin{equation*}
	\lim_{x\rightarrow\infty}\frac{W_0\left(x\right)}{\log\left(x\right)}=1 \ , \ \lim_{x\rightarrow0^{-}}\frac{W_{-1}\left(x\right)}{\log\left(-x\right)}=1.
	\end{equation*}
	
	Observe that since $h'$ is monotone decreasing, when $-2<x<0$, $\lambda$ has to be negative hence by \eqref{star} and the fact that $\frac{1}{2}{\left(x+2\right)}\leq1$, $W$ has to be less than $-1$ so in this case we'll choose the branch $W_{-1}$. By the same argument when $x\geq0$ we'll choose $W_0$. Let us denote the solution according to the branches $\lambda_{0/-1}(x)$ and overall we get-
	\begin{equation*}
	\Lambda^*\left(x\right)=\begin{cases}
	+\infty & x<-2,\\
	\frac{\pi^2}{3} & x=-2,\\
	h\left(\lambda_{-1}\left(x\right),x\right) & -2<x<0,\\
	h\left(\lambda_0\left(x\right),x\right) & x\geq0.\\
	\end{cases}
	\end{equation*}  
	It holds that
	\begin{align*}
	\lim_{x\downarrow-2}\Lambda^*\left(x\right)
	=\lim_{x\downarrow-2} \ \left(x+2\right)\left(\frac{1}{2}\left(x+2\right)+W_{-1}\left(-\frac{1}{2}\left(x+2\right)e^{-\frac{1}{2}{\left(x+2\right)}}\right)\right) &\\+2\Li_2\left(1-\exp\left(\frac{1}{2}\left(x+2\right)+W_{-1}\left(-\frac{1}{2}\left(x+2\right)e^{-\frac{1}{2}\left(x+2\right)}\right)\right)\right)
	& =\frac{\pi^2}{3},
	\end{align*}
	and 
	\begin{align*}
	\lim_{x\rightarrow+\infty}\Lambda^*\left(x\right)
	=\lim_{x\rightarrow+\infty} \ \left(x+2\right)\left(\frac{1}{2}{\left(x+2\right)}+W_{0}\left(-\frac{1}{2}{\left(x+2\right)}e^{-\tfrac{1}{2}{\left(x+2\right)}}\right)\right)&\\+2\Li_2\left(1-\exp\left({\frac{1}{2}{\left(x+2\right)}+W_{0}\left(-\frac{1}{2}{\left(x+2\right)}e^{-\frac{1}{2}{\left(x+2\right)}}\right)}\right)\right)
	& =+\infty,
	\end{align*}
	where the asymptotic of the functions $W_0,W_{-1},\Li_2$ is used for both limits.
	Note that $\Lambda^*\left(0\right)=0$ and that the function $\Lambda^*$ is non-negative in a neighborhood of $0$ hence by convexity, $0$ is the minimum of this function (see \Cref{fig: rate}), completing the proof of \Cref{lemma: alpha=1} and of \Cref{Theorem: babycase}.
	
\end{proof}
\section{Measurability of $\beta^*$}\label{measurabiliy}
Recall that $\beta^*:\Omega\rightarrow[0,2\pi)$ is such that 
\begin{equation*}
\frac{1}{N}\sum_{j=1}^{N}\log\left|f\left(r_0e^{i\theta_j}\cdot e^{i\beta^*}\right)\right|=\int_{r_0\mathbb{T}}\log\left|f\right|d\mu  
\end{equation*} 
when $f\neq0$ otherwise we set $\beta^*=0$.  
We understand our probability space $\Omega$ as follows. Consider the complex plane $\mathbb{C}$ with $\mathbb{C}_\nu$ the Borel sigma-algebra of $\mathbb{C}$. Then, we take the countable product measure space $\mathbb{C}^{\infty}_\nu$, with the sigma algebra generated by cylinder sets, i.e. $\mathbb{C}^{j-1}\times U_j\times \mathbb{C}^{n-j}$  where
$U_j\in\mathcal{B}\left(\mathbb{C}\right)$, $n\geq1$.     
Denote by $\text{A}\left(\mathbb{D}\right)$ - the space of analytic functions defined on the unit disk. Endow this space with the topology of uniform convergence on compact sets. This makes it a complete separable metric space. 
Now, $\beta^{*}:\Omega\mapsto[0,2\pi)$ can be viewed as a composition of two measurable functions.  
The map $\omega=\{\omega_n\}_{n=1}^{\infty}\mapsto f\left(z,\omega\right)=\sum_{n=1}^{\infty}\xi_n\left(\omega\right)z^{n}$ where $\xi_n\left(\omega\right)=\omega_n$ is measurable as a limit of measurable functions (projections are continuous).
Then, we need to prove that the map  $\beta^{*}:A\left(\mathbb{D}\right)\mapsto[0,2\pi]$ defined momentarily is measurable. 

Consider all $\beta$'s such that
\begin{equation*}
\frac{1}{N}\sum_{j=1}^{N}\log\left|f\left(r_0e^{i\theta_j}\cdot e^{i\beta}\right)\right|=\int_{r_0\mathbb{T}}\log\left|f\right|d\mu=:c_f
\end{equation*} 
(such $\beta$ exists as we've already mentioned by continuity, \Cref*{estimate:avergae}). We define 
$\beta^*\left(f\right):\text{A}\left(\mathbb{D}\right)\rightarrow\left[0,2\pi\right]$  as follows: 
\begin{equation*}
\beta^{*}\left(f\right)=\inf\{\beta\in[0,2\pi) \ : \ 	\frac{1}{N}\sum_{j=1}^{N}\log\left|f\left(r_0e^{i\theta_j}\cdot e^{i\beta}\right)\right|=c_f\}.
\end{equation*}
To see that $\beta^*$ is lower semi-continuous (and hence measurable), we want to show that for all $a\in[0,2\pi)$, $\{f\in\text{A}\left(\mathbb{D}\right) \ : \ \beta^{*}\left(f\right)>a\}\in \mathcal{B}_{\text{A}\left(r\bar{\mathbb{D}}\right)}$
where $\mathcal{B}_{\text{A}\left(r\bar{\mathbb{D}}\right)}$ is the Borel sigma-algebra with respect to the sup metric.
(since $r<1$ is fixed, we can restrict our domain to $r\mathbb{D}$.)
Indeed, 

\begin{align*}
\{f\in\text{A}&\left(\mathbb{D}\right) \ : \ \beta^{*}\left(f\right)>a\} \\ &=\bigg\{f\in\text{A}\left(\mathbb{D}\right) \ : \ \forall \beta\in\left[0,a\right] \ , \  \frac{1}{N}\sum_{j=1}^{N}\log\left|f\left(r_0e^{i\theta_j}\cdot e^{i\beta}\right)\right|\neq c_f\bigg\}\\ &=\bigcup_n \, \bigg\{f\in\text{A}\left(\mathbb{D}\right) \ : \ \forall \beta\in\left[0,a\right] \ , \  \Big|\frac{1}{N}\sum_{j=1}^{N}\log\big|f\left(r_0e^{i\theta_j}\cdot e^{i\beta}\right)\big|-c_f\Big|>\frac{1}{n}\bigg\}
\end{align*}

So if $\sup_{r\bar{\mathbb{D}}}\left|f-g\right|<\delta$ then (by continuity) the integrals are sufficiently close to each other $\left|c_f-c_g\right|<\epsilon$ hence this set is open.

\end{document}